\documentclass[a4paper,12pt]{amsart}

\usepackage{amsmath,amssymb,amsfonts,amsthm,amsopn}
\usepackage{graphicx,color}

\usepackage{amsmath}

\usepackage{float}
\usepackage{caption}
\usepackage{geometry}
\newtheorem{theorem}{Theorem}[section]
\newtheorem{lemma}[theorem]{Lemma}

\theoremstyle{remark}

\newtheorem*{remark*}{Remark}

\theoremstyle{definition}

\def\bC{\mathbb{C}}

\newcommand\Aa{{\mathcal{A}_\alpha}}

\numberwithin{equation}{section}
\title[Weighted contractivity for derivatives in Bergman space]{Weighted contractivity for derivatives of functions in the Bergman space on the unit disk}
\begin{document}

\makeatletter
\@namedef{subjclassname@2020}{%
  \textup{2020} Mathematics Subject Classification}
\makeatother
\numberwithin{equation}{section}
\title[Weighted contractivity for derivatives in Bergman space]{Weighted contractivity for derivatives of functions in the Bergman space on the unit disk}


\author{David Kalaj}
\address{University of Montenegro, Faculty of Natural Sciences and
Mathematics, Cetinjski put b.b. 81000 Podgorica, Montenegro}
\email{davidk@ucg.ac.me}

\author{Petar Melentijevi\'c}
\address{University of Belgrade, Faculty of Mathematics, Studentski trg 16, 11000, Beograd}
\email{petar.melentijevic@matf.bg.ac.rs}
\thanks{The second author is partially supported by MPNTR grant 174017, Serbia.}

\begin{abstract} In a recent paper \cite{ramos}, Ramos and Tilli proved certain sharp inequality for analytic functions in subdomains of the unit disk. We will generalize their main inequality for derivatives of functions from Bergman space with respect to two different measures. Some connections with an analog for the Fock spaces, earlier investigated in \cite{Kalaj}, will also be discussed.

\end{abstract}
\maketitle

\section{Introduction}

\subsection{Bergman spaces on $\bC^+$ and $\mathbb{D}$}For every $\alpha>-1$,
the Bergman space $\Aa(\mathbb{D})=\Aa$ of the disc is the Hilbert space of all functions
$f:\mathbb{D}\to \bC$ which are holomorphic in the unit disk $\mathbb{D}$ and satisfying
\[
\Vert f\Vert_\Aa^2 := \int_{\mathbb{D}} |f(z)|^2 (1-|z|^2)^{\alpha} \,dA(z) <+\infty.
\]

The weighted $L^2$ norm defining this space is induced by the scalar product
\[
\langle f,g\rangle_\alpha := \int_{\mathbb{D}} f(z)\overline{g(z)} (1-|z|^2)^{\alpha}\, dA(z).
\]
Here and throughout, $dA(z)=dxdy$ denotes the two-dimensional Lebesgue measure on $\mathbb{D}$.

Given $w\in \mathbb{D}$, the reproducing kernel relative to $w$, i.e. the (unique) function
$K_w\in\Aa$ such that
\begin{equation}
\label{repker}
f(w)=\langle f,K_w\rangle_\alpha\quad\forall f\in\Aa,
\end{equation}
is given by
\[
K_w(z):=  \frac {\alpha+1}\pi(1-\overline{w}z)^{-\alpha-2},\quad z\in \mathbb{D};
\]
 note that $K_w\in\Aa$.  We refer the reader to \cite{hakan, Seip} and the references therein for further meaningful properties in the context of Bergman spaces.

Recently, in the papers \cite{Kulikov, NicolaTilli} a new method for studying distribution for analytic functions in hyperbolic and Euclidean case was discovered. Some extensions of these results are  given in \cite{Kalaj, Kalaj1, OrtegaCerda, Frank}, while in \cite{Llineares, Petar1}, the reader can find (partial) solutions of some other problems where the theorems from \cite{Kulikov} are used. Here, we will turn our attention to an another paper in which these methods are used, \cite{ramos}, where the Radon measure
\[
\mu(\Omega):=\int_\Omega (1-|z|^2)^{-2}dA(z),\quad\Omega\subseteq \mathbb{D},
\]
on the disc $\mathbb{D}$ is considered. This measure is, in fact, the area measure in the usual Poincar\'e model of the hyperbolic space (up to a multiplicative factor 4). 

The authors of \cite{ramos} maximized the quantity
\begin{equation}\label{eq:optimal-bergman-object}
R(f,\Omega) =
\frac{\int_\Omega |f(z)|^2 (1 - |z|^2)^{\alpha+2} \, d\mu(z)}{\Vert f \Vert_\Aa^2},
\end{equation}
over all \( f \in \Aa \) and \( \Omega \subset \mathbb{D} \) with \( \mu(\Omega) = s/4 \), because this approach exploits the equivalence between the norm in the Hardy space \( H^2(\mathbb{C}^+) \) and the transformed Bergman norm. Specifically, the functional \( R(f, \Omega) \) encapsulates the interaction of the Bergman weight \( (1 - |z|^2)^{\alpha+2} \) and the geometry of \( \Omega \), while being normalized by \( \Vert f \Vert_\Aa^2 \), the square of the \( \Aa \)-norm of \( f \).

The maximization aligns with the property that
\[
\Vert f \Vert_{H^2(\mathbb{C}^+)}^2 = \Vert T_{\alpha}(B_{\alpha})f \Vert_\Aa^2,
\]
where \( B_{\alpha} \) is the Bergman transform of order \( \alpha \), defined as
\[
B_{\alpha}f(z) = 2^{\alpha}\Gamma(\alpha+1) \int_{0}^{+\infty} t^{\frac{\alpha+1}{2}} \hat{f}(t) e^{i zt} \, dt,
\]
and \( T_{\alpha} \) is a transformation that maps functions into a weighted Bergman space, given by
\[
T_{\alpha}f(z) = \frac{2^{\frac{\alpha}{2}}}{(1 - z)^{\alpha+2}} f\left(\frac{z+1}{i(z-1)}\right).
\]

This formulation demonstrates how \( R(f, \Omega) \) not only reflects the structure of \( \Aa \) and \( \Omega \), but also ensures compatibility with the integral transform framework, thereby offering an optimal functional for analyzing Bergman-type spaces under geometric constraints on \( \Omega \).

More precisely, the main result of Ramos and Tilli in \cite{ramos} is the following theorem.
\begin{theorem}\label{thm:main-bergman} Let $\alpha>-1,$ and $s>0$ be fixed. Among all functions $f\in \Aa$ and among
all measurable sets $\Omega\subset \mathbb{D}$ such that $\mu(\Omega)=s$, the quotient $R(f,\Omega)$ as defined in \eqref{eq:optimal-bergman-object} satisfies the inequality
\begin{equation}\label{eq:upper-bound-quotient}
	R(f,\Omega) \le R(1,\mathbb{D}_s),
\end{equation}
where $\mathbb{D}_s$ is a disc centered at the origin with $\mu(\mathbb{D}_s) = s.$ Moreover, there is equality in \eqref{eq:upper-bound-quotient} if and only if $f$ is a multiple of some
reproducing kernel $K_w$ and $\Omega$ is a ball centered at
$w$, such that $\mu(\Omega)=s$.
\end{theorem}
Note that, in the Poincar\'e disc model in two dimensions, balls in the pseudohyperbolic metric coincide with Euclidean balls, but the Euclidean and hyperbolic centers differ in general, as well as the respective radii.

The proof of Theorem~\ref{thm:main-bergman} is based on the isoperimetric inequality for hyperbolic metric on the unit disk and the point-wise estimate (\cite{Vukotic})
\begin{equation}
\label{pointwise}
|f(z)|^2\le \frac {\alpha+1}\pi (1-|z|^2)^{-{(\alpha+2)}}\|f\|_{\Aa}^2, \ \ z\in\mathbb{D}.
\end{equation}

The aim  of this paper is to generalize this result for the higher order derivatives of holomorphic functions. In order to do so, define the Gaussian hypergeometric function $$F(a,b;c;z)=\sum_{n=0}^\infty \frac{(a)_n (b)_n}{(c)_n n!} z^n,$$ where $(d)_n=d(d+1)\dots (d+n-1)$ is the Pochhammer symbol.

We begin by the following lemma (an analog of \eqref{pointwise} for derivatives):

\begin{lemma}
Let $$g(|z|)= ({\alpha+2})_n \cdot n! \cdot (1-|z|^2)^{-\alpha-2-2n}F\left(-1-\alpha-n,-n;1;r^2\right).$$
	If $f\in \Aa$, then $$|f^{(n)}(z)|^2\le g(|z|)\|f\|_{A_\alpha}^2.$$ Here $\|f\|_{A_\alpha}^2:=\frac{\alpha+1}{\pi}\|f\|_{\Aa}^2.$
\end{lemma}
\begin{proof}
	From reproducing formula for functions from the Bergman space
	$$  f(z)=\int_{\mathbb{D}}\frac{f(w)}{(1-z\overline{w})^{{\alpha+2}}}dA_\alpha(w),$$ where $$dA_\alpha(w)=\frac{\alpha+1}{\pi}(1-|w|^2)^{{\alpha}},$$
	we find
	$$f^{(n)}(z)=\int_{\mathbb{D}}\frac{f(w)(\alpha+2)(\alpha+3)\cdots(n+\alpha+1)\overline{w}^n}{(1-z\overline{w})^{n+2+\alpha}}dA_{\alpha}(w)$$
	and consequently, by Cauchy-Schwarz inequality:
	$$|f^{(n)}(z)|\leq (\alpha+2)(\alpha+3)\cdots(n+{\alpha+1})\|f\|_{A_{\alpha}}\sqrt{\int_{\mathbb{D}}\frac{|w|^{2n}}{|1-z\overline{w}|^{2{(\alpha+2)}+2n}}dA_{\alpha}(w)}. $$
	The integral can be calculated using polar coordinates and Parseval's formula, thus giving:
	\[\begin{split}\frac{\alpha+1}{\pi}&\int_{\mathbb{D}}|1-z\overline{w}|^{-2{(\alpha+2)}-2n}|w|^{2n}(1-|w|^2)^{{\alpha}}dA(w)
	\\&=2({\alpha+1}) \int_{0}^{1}\rho^{2n+1}(1-\rho^2)^{{\alpha}}\bigg(\frac{1}{2\pi}\int_{0}^{2\pi}|1-z\rho e^{\imath\theta}|^{-2{(\alpha+2)}-2n}d\theta\bigg)d\rho
	\\&=2({\alpha+1})\sum_{k=0}^{+\infty}\binom{{\alpha+1}+n+k}{k}^2|z|^{2k}\int_{0}^{1}\rho^{2n+2k+1}(1-\rho^2)^{{\alpha}}d\rho
\\&=\frac{n!}{({{\alpha+2}})_n}F(n+1,{\alpha+2}+n;1;|z|^2).\end{split}\]
	Hence,
	$$|f^{(n)}(z)|^2\leq n!({{\alpha+2}})_n F(n+1,{\alpha+2}+n;1;|z|^2)\|f\|^2_{A_\alpha},$$
	and by, Euler's transformation, the last hypergeometric function is equal to $(1-|z|^2)^{-2n-{(\alpha+2)}}F(-n,-{\alpha}-n-1;1;|z|^2).$
	\end{proof}
Similar estimates for the first and higher order derivatives for functions in weighted Bergman and Fock spaces can be found in \cite{DavidDjordjije, Petar}.

Now define two weighted measures $$d\mu_{n,\alpha}(z)=\frac{({\alpha+1}+2n)(1-|z|^2)^{{\alpha+2}+2n}{d\mu(z)}}{\pi ({{\alpha+2}})_n \cdot n! \cdot F\left(-1-{\alpha}-n,-n;1;|z|^2\right)}$$
and
$$d\nu_{n,\alpha}(z)=\frac{\Gamma({{\alpha+2}})(1-|z|^2)^{{\alpha+2}+2n}{d\mu(z)}}{\pi \Gamma({\alpha+1}+2n) }$$
and consider the space of holomorphic functions $f$ defined on the unit disk, so that $$\|f\|_{n,\alpha}^2:=\int_{\mathbb{D}} |f(z)|^2 d\mu_{n,\alpha}(z)<\infty.$$ Denote this space by $\mathcal{A}_{n,\alpha}.$ Observe that, for $n=0$, $\mathcal{A}_{n,\alpha}=\mathcal{A}_{\alpha}.$

The motivation for considering the first of them comes from the pointwise estimate \eqref{pointwise}. Since
every measure of the form $C (1-|z|^2)^{{\alpha+2}+2n}$ on the unit disk $\mathbb{D}$ is equivalent to $d\mu_{n,\alpha},$ it also make sense to consider $d\nu_{n,\alpha}.$ At the end of our paper we will see that, after certain limitting process, the results for the measure $d\mu_{n,\alpha}$ can be transferred to the appropriate results for Fock spaces, earlier proved in \cite{Kalaj}.

Now we formulate the following extension of Theorem~\ref{thm:main-bergman} for $n\ge 1$.
\begin{theorem}\label{kalaj}
Assume that $f\in \mathcal{A}_{\alpha}$ and define the $n-$quotient $$R_{n,\mu}(f,\Omega)=\frac{\int_{\Omega} |f^{(n)}(z)|^2 d\mu_{n,{\alpha}}}{\|f\|^2_{\mathcal{A}_{\alpha}}}$$
and the analogous $R_{n,\nu}(f,\Omega)$ with the measure $d\nu_{n,{\alpha}}.$ Then we have the following sharp inequalities
\begin{equation}
\label{rqomu}
R_{n,\mu}(f,\Omega)\le \left(1-(1+s/\pi)^{1-X}\right), \quad \text{for}\quad 1\leq n \leq 3,
\end{equation}
\begin{equation}
\label{rqonu}
R_{n,\nu}(f,\Omega)\le \left(1-(1+s/\pi)^{1-2n-{(\alpha+2)}}\right), \quad \text{for}\quad n\geq 1,
\end{equation}
where $\mu(\Omega)=s$ and $X= (n+1)(n+2+\alpha).$
The equality can not be attained in any of these two estimates.
\end{theorem} This question is well-motivated regarding recent result by Kalaj, where similar estimate is proved for Fock spaces. In fact, we will see that after a limitting argument his estimate follows from ours, i.e. from the case of the weighted Bergman spaces. Let us, also, say that, in \cite{AbreuSpeckbacher}, the authors has mentioned that they expect certain Bergman space analogs to Kalaj's estimates which strongly depend on some special inequalities with Jacobi polynomials. This has turned out to be true and we will come back to this point in Lemma 2.1.

We believe that \eqref{rqomu} holds for each $n\geq 1,$ but we were not able to prove this.

\section{Proof of Theorem~\ref{kalaj}}

In this section we will prove the Theorem \ref{kalaj}. In order to do this,
we need the following lemmas
\begin{lemma} \label{petar}For $\alpha>-1$ and $n\in \mathbb{N},$ we have
	$$\int_{\mathbb{D}}|f^{(n)}(z)|^2d\mu_{n,\alpha}(z)\le  ({\alpha+1})\int_{\mathbb{D}}|f(z)|^2 (1-|z|^2)^{{\alpha}}\frac{dxdy}{\pi}$$
	and the analogous with $d\nu_{n,\alpha}$ in the place of $d\mu_{n,\alpha}$.
\end{lemma}
We postpone the proof of this lemma for the next section.
\begin{lemma}\label{petarce:)}
For $n\ge 0$ and $\alpha>-1,$ we have $$\Delta \log g(|z|)\le 4 (n+1)(n+2+\alpha)(1-|z|^{-2})^{-2}.$$
\end{lemma}
\begin{proof}
Let us denote $F(t)=F(n+1,n+2+\alpha;1;t)$. From the formula for laplacian in polar coordinates we find
\[
\begin{split}
\Delta \log g(\rho)&=\Delta \log F(\rho^2)=\frac{1}{\rho}\frac{\partial}{\partial \rho}\log F(\rho^2)+\frac{\partial^2}{\partial \rho^2}\log F(\rho^2)\\
&=\frac{4F'(\rho^2)F(\rho^2)+4\rho^2F(\rho^2)F''(\rho^2)-4\rho^2(F'(\rho^2))^2}{F^2(\rho^2)}.
\end{split}
\]
Hence the inequality we intend to prove is equivalent to
$$F(t)F'(t)+tF(t)F''(t)-t(F'(t))^2\leq (n+1)(n+2+\alpha)(1-t)^{-2}F^2(t).$$
This inequality follows from the next two inequalities:
$$(n+1)(n+2+\alpha)(1-t)^{-1}(F(t))^2\geq F'(t)F(t)$$
and
$$(n+1)(n+2+\alpha)((1-t)^{-2}-(1-t)^{-1})(F(t))^2\geq t(F(t)F''(t)-(F'(t))^2).$$
First of them, after using $F(a,b;c;t)=(1-t)^{c-a-b}F(c-a,c-b;c;t),$ is easily seen to be equivalent with $F(-n,-1-{\alpha}-n;1;t)\geq F(-n,-1-{\alpha}-n;2;t).$ The second inequality reduces to
$$(n+1)(n+2+\alpha)(1-t)^{-2}\geq \bigg(\frac{F'}{F}\bigg)'(t),$$
or, by using the same identity:
\[
\begin{split}
&\quad(1-t)^{-2}\geq \bigg((1-t)^{-1}\frac{F(-n,-1-{\alpha}-n;2;t)}{F(-n,-1-{\alpha}-n;1;t)}\bigg)'\\
&=(1-t)^{-2}\frac{F(-n,-1-{\alpha}-n;2;t)}{F(-n,-1-{\alpha}-n;1;t)}+(1-t)^{-1}\bigg(\frac{F(-n,-1-{\alpha}-n;2;t)}{F(-n,-1-{\alpha}-n;1;t)}\bigg)'.
\end{split}
\]
However, main results of the papers \cite{BIERNACKIKRZYZ, HVV} gives that $\frac{F(-n,-1-{\alpha}-n;2;t)}{F(-n,-1-{\alpha}-n;1;t)}$ is monotone decreasing and $\leq 1,$ which concludes the proof.
\end{proof}

Now, we turn to the proof of main theorem. Let $\|f\|_{\mathcal{A}_{\alpha}}=1.$ First we define
$$u_n(z) = \frac{|f^{(n)}(z)|^2}{\tilde{g}(|z|)}$$
with $\tilde{g}(|z|)=\frac{(\alpha+1) n!({{\alpha+2}})_nF(n+1,n+2+\alpha;1;|z|^2)}{\alpha+2n+1}$ and also
$$I_n(s) =  \int_{\{z: u_n(z)> u_n^\ast(s)\}} u_n(z) d\mu(z),$$ where $u_n^\ast(s)$ is the unique positive real number $t$ so that $\mu(\{z: u_n(z)>t\})=s$.

Observe that \begin{equation}\label{izero}I_n(0)=\int_{\{z: u_n(z)> u_n^\ast(0)\}} u_n(z) d\mu(z)=0,\end{equation} and \begin{equation}\label{iinf}\begin{split}I_n(+\infty)&=\int_{\{z: u_n(z)> u_n^\ast(+\infty)\}} u_n(z) d\mu(z)\\&=\int_{\{z: u_n(z)> 0\}} u_n(z) d\mu(z)=\int_{\mathbb{D}}u_n(z) d\mu(z)\le 1\end{split}.\end{equation}
Note that the last inequality follows from Lemma 2.1.

In a similar way as in \cite{ramos} we can prove
\begin{lemma}\label{thm:lemma-derivatives} The function $\varrho(t) := \mu(\{z: u_n(z) > t\})$ is absolutely continuous on $(0,\max u_n],$ and
	\[
	-\varrho'(t) = \int_{\{z: u_n(z) = t\}} |\nabla u_n|^{-1} (1-|z|^2)^{-2} \, d \mathcal{H}^1.
	\]
	In particular, the function $u_n^*$ is, as the inverse of $\varrho,$ locally absolutely continuous on $[0,+\infty),$ with
	\[
	-(u_n^*)'(s) = \left( \int_{\{z: u_n(z)=u_n^*(s)\}} |\nabla u_n|^{-1} (1-|z|^2)^{-2} \, d \mathcal{H}^1 \right)^{-1}.
	\]
\end{lemma}
Let us then denote the boundary of the superlevel set where $u > u^*(s)$ as
\[
A_s=\partial\{z:u_n(z)>u_n^*(s)\}.
\]
We then imitate the corresponding proof in \cite{ramos}. By Lemma \ref{thm:lemma-derivatives},
\[
I_n'(s)=u_n^*(s),\quad
I_n''(s)=-\left(\int_{A_s} |\nabla u_n|^{-1}(1-|z|^2)^{-2}\,d{\mathcal H}^1\right)^{-1}.
\]
Now we apply Cauchy-Schwarz inequality to get
\[
\left(\int_{A_s} |\nabla u_n|^{-1}(1-|z|^2)^{-2}\,d{\mathcal H}^1\right)
\left(\int_{A_s} |\nabla u_n|\,d{\mathcal H}^1\right)
\geq
\left(\int_{A_s} (1-|z|^2)^{-1}\,d{\mathcal H}^1\right)^2,
\]
letting
\[
L(A_s):= \int_{A_s} (1-|z|^2)^{-1}\,d{\mathcal H}^1
\]
denote the length of $A_s$ in the hyperbolic metric, we obtain the lower bound
\begin{equation}\label{eq:lower-bound-second-derivative}
I_n''(s)\geq - \left(\int_{A_s} |\nabla u_n|\,d{\mathcal H}^1\right)
L(A_s)^{-2}.
\end{equation}
To determine the first term in the product on the right-hand side of \eqref{eq:lower-bound-second-derivative}, we observe that
\[\begin{split}\Delta \log u_n(z) &=\Delta \log \frac{|f^{(n)}(z)|^2}{\tilde{g}(|z|)}\\&=\Delta \log {|f^{(n)}(z)|^2 }-\Delta\log \tilde{g}(|z|)\\&\geq-4(n+1)(n+2+\alpha)(1-|z|^2)^{-2},
\end{split}\]
which then implies that, letting $w(z)=\log u_n(z)$,
\[\begin{split}
\frac {-1} {u_n^*(s)} \int_{A_s} |\nabla u_n|\,d{\mathcal H}^1
 &=\int_{A_s} \nabla w\cdot\nu \,d{\mathcal H}^1
\\&=\int_{\{z: u_n(z)>u_n^*(s)\}} \Delta w\,dA(z)
\\&\ge-4 (n+1)(n+2+\alpha)\int_{\{z: u_n(z)>u_n^*(s)\}}
(1-|z|^2)^{-2}\,dA(z)
\\& =-4 (n+1)(n+2+\alpha) \mu(\{z: u_n(z)>u_n^*(s)\})\\&=
-4 (n+1)(n+2+\alpha)s.
\end{split}
\]
Therefore,
\begin{equation}\label{eq:lower-bound-second-almost}
I_n''(s)\geq -4(n+1)(n+2+\alpha)s u_n^*(s)L(A_s)^{-2}=
-4(n+1) (n+2+\alpha)s I_n'(s)L(A_s)^{-2}.
\end{equation}
On the other hand, the isoperimetric inequality for the hyperbolic metric gives
\[
L(A_s)^2 \geq 4\pi s + 4 s^2,
\]
so that, plugging into \eqref{eq:lower-bound-second-almost}, we obtain
\begin{equation}\begin{split}\label{eq:final-lower-bound-second}
I_n''(s)&\geq -4 (n+1)(n+2+\alpha)s I_n'(s)(4\pi s+4 s^2)^{-1}
\\&=-(n+1)(n+2+\alpha) I_n'(s)(\pi+s)^{-1}.\end{split}
\end{equation}

Then we obtain $$I_n''(s)  +(n+1) (n+2+\alpha) I_n'(s)(\pi+s)^{-1}\ge 0.$$ Thus $$J(t) =  I_n(T(t)),$$ is convex, where $$T(t) = -\pi +\pi\left(1-t\right)^{\frac{1}{1-(n+1)(n+2+\alpha)}}.$$

Since $T(0)=0$ and $T(1)=+\infty$, it follows that $J(0)=0$ and $J(1)\le 1$. Thus $J(s)\le s$. In other words,
$$I_n(s) \le  \theta_n(s)$$ where
$$\theta_n(s) = 1-(1+s/\pi)^{1-X}=T^{-1}(s)$$ is the inverse function of $T$ with $X=(n+1)(n+2+\alpha).$

In the case of measure $d\nu_{n,\alpha},$ we define $u_n,$ as above, but with $$g(|z|)=\frac{ \Gamma(\alpha+1+2n)}{(\alpha+1)\Gamma(\alpha+2)}(1-|z|^2)^{-\alpha-2n-2}.$$ Proceeding as earlier, we have
$$\Delta \log u_n(z)\geq -4(2n+2+\alpha)(1-|z|^2)^{-2},$$
thus getting
$$ \frac {-1} {u_n^*(s)} \int_{A_s} |\nabla u_n|\,d{\mathcal H}^1\geq -4(2n+2+\alpha)s$$
and
$$I_n''(s)\geq -4(2n+2+\alpha)sI_n'(s)L(A_s)^{-2},$$
while the isoperimetric inequality for the hyperbolic metric implies
$$I_n''(s)\geq -(2n+2+\alpha)I_n'(s)(\pi+s)^{-1}.$$
Now it is evident that our estimate holds with $X=2n+2+\alpha.$

\section{A proof of Lemma 2.1}

Let us first prove the Lemma 2.1 in the case of the measure $d\nu_{n,\alpha},$ since its proof is significantly easier.

Using Taylor expansion and Parseval's identity we get the following equivalent form:
\[
\begin{split}
&\sum_{k=n}^{+\infty}\frac{k^2(k-1)^2\cdots(k-n+1)^2|a_k|^2\Gamma({{\alpha+2}})}{\Gamma({\alpha+1}+2n)}\int_{0}^{1}2\rho^{2k-2n+1}(1-\rho^2)^{{{\alpha}}+2n} d\rho\\
&\leq \sum_{k=0}^{+\infty}\frac{k!\Gamma({{\alpha+2}})}{\Gamma(k+{{\alpha+2}})}|a_k|^2.
\end{split}
\]
This is easily seen to be equivalent with
$$c_{k,n}=\frac{k!}{(k-n)!}\frac{\Gamma(k+{{\alpha+2}})}{\Gamma(k+n+2+\alpha)}\leq 1,$$
which follows from the observations that $\frac{c_{k+1,n}}{c_{k,n}}=\frac{(k+1)(k+{{\alpha+2}})}{(k-n+1)(k+n+2+\alpha)}\geq 1$ and $\lim_{k\rightarrow +\infty}c_{k,n}=1.$

The case of the measure $d\mu_{n,\alpha}$ is much harder and we will prove it only for $n\leq 3.$

Similarly as above, we reduce it to the inequality:
\[
\begin{split}
&\quad\sum_{k=n}^{+\infty}\frac{k^2(k-1)^2\cdots(k-n+1)^2|a_k|^2}{n!({{\alpha+2}})_n}\int_{0}^{1}\frac{2\rho^{2k-2n+1}(1-\rho^2)^{{{\alpha}}+2n}}{F(-1-{\alpha}-n,-n;1;\rho^2)} d\rho\\
&=\sum_{k=n}^{+\infty}\frac{k^2(k-1)^2\cdots(k-n+1)^2|a_k|^2}{n!({{\alpha+2}})_n}\int_{0}^{1}\frac{t^{k-n}(1-t)^{{{\alpha}}+2n}}{F(-1-{\alpha}-n,-n;1;t)} d t\\
&\leq \frac{1}{{\alpha+1}+2n}\sum_{k=0}^{+\infty}\frac{k!\Gamma({{\alpha+2}})}{\Gamma(k+{{\alpha+2}})}|a_k|^2.
\end{split}
\]
Hence, it is sufficient (and necessary) to prove that
\[
\begin{split}
\label{koeficijenti}
&\frac{k^2(k-1)^2\cdots(k-n+1)^2}{n!({{\alpha+2}})_n}\int_{0}^{1}\frac{t^{k-n}(1-t)^{{\alpha}+2n}}{F(-1-\alpha-n,-n;1;t)} d t\\
&\leq \frac{1}{{\alpha+1}+2n}\frac{k!\Gamma({\alpha+2})}{\Gamma(k+{\alpha+2})} \quad \text{for} \quad k\geq n.
\end{split}
\]
Denoting $P(t)=F(-1-\alpha-n,-n;1;t)$ and reformulating our inequality we stand at
\[
\begin{split}
&\frac{1}{B(k-n+1,{\alpha+1}+2n)}\int_{0}^{1}t^{k-n}(1-t)^{{\alpha}+2n}P(t)^{-1}dt\\
&\leq \frac{n!\Gamma(k+{\alpha+2}+n)\Gamma({\alpha+2}+n)}{\Gamma({\alpha+2}+2n)\Gamma(k+{{\alpha+2}})k(k-1)\cdots(k-n+1)}
\end{split}
\]
We will rewrite $\frac{1}{P(t)}$ using partial fractions as $$\frac{1}{P(t)}=\frac{1}{(1+\beta_1 t)(1+\beta_2 t)\cdots(1+\beta_n t)}=\frac{A_1}{1+\beta_1t}+\frac{A_2}{1+\beta_2t}+\dots+\frac{A_n}{1+\beta_n t},$$
for some $A_j \in \mathbb{R},  j \in \{1,2,\dots,n\}$ and $\beta_j \in \mathbb{R}^{+},  j \in \{1,2,\dots,n\}.$
This is possible, since by the formula from \cite{ABRAMOWITZSTEGUN}, $P(t)=(1-t)^nP_n^{(0,{\alpha+1})}(\frac{1+t}{1-t}),$ where $P_n^{(0,{\alpha+1})}(x)$ are the corresponding Jacobi polynomials and each $P_n^{(0,{\alpha+1})}(x)$ has exactly $n$ simple real zeros (this is general theorem on orthogonal polynomials). Also, $P(t)$ has only negative zeros because all its coefficients are positive.
From
$$1=P(t)\sum_{j=1}^{n}\frac{A_j}{1+\beta_j t}= \sum_{j=1}^{n}A_j (1+\beta_1 t)\cdots(1+\beta_{j-1}t)(1+\beta_{j+1}t)\cdots(1+\beta_nt)                                      $$
taking the limit when $t\rightarrow-\frac{1}{\beta_j},$ we find that
$$A_j=\beta_j^{n-1}\prod_{i\neq j} (\beta_j-\beta_i)^{-1}.  $$
Expansion into partial fractions reduces to the calculation of the integral
$$\int_{0}^{1}t^{k-n}(1-t)^{{\alpha}+2n}(1+\gamma t)^{-1} dt $$
which after the substitution $z=\frac{(\gamma+1)t}{1+\gamma t}$ becomes
\[
\begin{split}
&(\gamma+1)^{-k+n-1}\int_{0}^{1}z^{k-n}(1-z)^{{\alpha}+2n}
\big(1-\frac{\gamma}{\gamma+1}z\big)^{-k-{\alpha-n-1}} dz\\
&=\frac{B(k-n+1,{\alpha+1}+2n)}{(\gamma+1)^{k-n+1}}F\big(k-n+1,n+k+{\alpha+1};k+{\alpha+2}+n;\frac{\gamma}{\gamma+1}\big),
\end{split}
\]
by the Euler integral representation for hypergeometric function.
Hence, we have:
\[
\begin{split}
&\frac{\Gamma(k-n+1)\Gamma({\alpha+1}+2n)}{\Gamma(k+n+2+\alpha)(\gamma+1)^{k-n+1}}F\big(k-n+1,n+k+{\alpha+1};k+{\alpha+2}+n;\frac{\gamma}{\gamma+1}\big)\\
&=\frac{(n+k+{\alpha}+3)\Gamma(k-n+1)\Gamma({\alpha+1}+2n)}{\Gamma(k+n+2+\alpha)(\gamma+1)^{-2n-{\alpha}}}\int_{0}^{1}\frac{y^{k+{\alpha}+n}}{\big(1+\gamma y \big)^{1+\alpha+2n}}dy,
\end{split}
\]
where the last equality follows from the simple change of variable.
Now, the main inequality can be rewritten as
\[
\begin{split}
&\sum_{j=1}^{n}A_j(1+\beta_j)^{{\alpha}+2n}\int_{0}^{1}\frac{y^{k+{\alpha}+n}}{\big(1+\beta_j y\big)^{{\alpha+1}+2n}}dy\\
&\leq \frac{n!\Gamma(n+2+\alpha)}{\Gamma(2n+2+\alpha)}\frac{(k+{\alpha+2})_{n-1}}{(k-n+1)_n}.
\end{split}
\]

Denote $$\varphi(y)=\sum_{j=1}^{n}B_jy^{k-j}$$ and $$\psi(y)=\frac{y^{k+{\alpha}+n}}{\varphi(y)}\sum_{j=1}^{n}\frac{A_j(1+\beta_j)^{{\alpha}+2n}}{\big(1+\beta_j y\big)^{{\alpha+1}+2n}},$$
where $B_j$ is chosen so that
$$\int_{0}^{1}\varphi(y) d y=\sum_{j=1}^{n}\frac{B_j}{k-j+1}=\frac{(k+{\alpha+2})_{n-1}}{(k-n+1)_n}.$$
From
$$(x+{\alpha+2})_{n-1}=\sum_{j=1}^{n}\frac{B_j(x-n+1)_n}{x-j+1},$$
by letting $x \rightarrow j-1,$ we get
\[
\begin{split}
(j+{\alpha+1})_{n-1}&=B_j(j-1)(j-2)\cdots1\cdots(-1)(-2)\cdots(j-n)\\
&=(-1)^{n-j}(j-1)!(n-j)!B_j,
\end{split}
\]
i.e. $$B_j=(-1)^{n-j}\frac{(j+{\alpha+1})_{n-1}}{(j-1)!(n-j)!}.$$

Our inequality reads as
$$\int_{0}^{1}\varphi(y)\psi(y) dy \leq  \frac{n!\Gamma(n+2+\alpha)}{\Gamma(2n+2+\alpha)}\frac{(k+{\alpha+2})_{n-1}}{(k-n+1)_n},$$therefore, it will be enough to prove that
$$\psi(y)\leq \psi(1)=\frac{n!\Gamma(n+2+\alpha)}{\Gamma(2n+2+\alpha)}.$$
Let us explain the last equality. By their definitions, we have
$$\psi(1)=\frac{1}{\varphi(1)}\sum_{j=1}^{n}A_j(1+\beta_j)^{-1}=\frac{\sum_{j=1}^{n}\frac{A_j}{1+\beta_j}}{\sum_{j=1}^{n}B_j}$$
$$=\frac{1}{P(1)\sum_{j=1}^{n}B_j}=\frac{1}{F(-1-{\alpha}-n,-n;1;1)\sum_{j=1}^{n}B_j}.$$
From the identity $(x+{\alpha+2})_{n-1}=\sum_{j=1}^{n}\frac{B_j(x-n+1)_n}{x-j+1},$ we see that $\sum_{j=1}^{n}B_j$ is the coefficient with $x^{n-1},$ on the right-hand side, which from the left-hand side is easily seen to be equal to $1$. Using Gauss' theorem $F(a,b;c;1)=\frac{\Gamma(c-a-b)\Gamma(c)}{\Gamma(c-a)\Gamma(c-b)}$ we get $F(-{\alpha-1}-n,-n;1;1)=\frac{\Gamma(2n+2+\alpha)}{n!\Gamma(n+2+\alpha)},$ thus finding the value $\psi(1).$

After substitution $y=1-x,$ this becomes
$$\sum_{j=1}^{n}\frac{A_j}{1+\beta_j}\bigg(1-\frac{\beta_j}{\beta_j+1}x\bigg)^{-\alpha-1-2n}\leq\quad\frac{n!\Gamma(n+2+\alpha)}{\Gamma(2n+2+\alpha)}
\sum_{j=1}^{n}B_j(1-x)^{-{\alpha}-n-j}.$$
Series expansion on $x$ reduces the problem to proving
$$\binom{{\alpha}+2n+l}{l}\sum_{j=1}^{n}\frac{A_j}{1+\beta_j}\bigg(\frac{\beta_j}{1+\beta_j}\bigg)^l\leq \frac{n!\Gamma(n+2+\alpha)}{\Gamma(2n+2+\alpha)}\sum_{j=1}^{n}B_j\binom{{\alpha}+n+j+l-1}{l}.$$
The sum on right-hand side can be calculated:
\[
\begin{split}
&\quad\sum_{j=1}^{n}B_j\binom{{\alpha}+n+j+l-1}{l}\\
&=\sum_{j=0}^{n-1}(-1)^{n-j-1}\frac{(j+{\alpha+2})_{n-1}}{j!(n-j-1)!}\binom{{\alpha}+n+j+l}{l}\\
&=\frac{\Gamma({\alpha}+l+n+1)}{(n-1)!l!\Gamma({\alpha+2})}\sum_{j=0}^{n-1}(-1)^{n-j-1}\binom{n-1}{j}\frac{({\alpha+1}+n+l)_j}{({\alpha+2})_j}\\
&=\frac{(n+l-1)!\Gamma({\alpha}+l+n+1)}{(l!)^2(n-1)!\Gamma({\alpha+2}+n-1)},
\end{split}
\]
using the next lemma:
\begin{lemma}
	$$\sum_{s=0}^{m}\frac{(-m)_s(\beta+l+m)_s}{(\beta)_{s}s!}=(-1)^m\frac{(l+1)_m}{(\beta)_m}.$$
\end{lemma}
\begin{proof}
	Note that both sides of the identity are $m-$degree polynomials on $l.$ Hence, it will be enough to prove it for $m+1$ different numbers $l.$ Taking $l=-t, $ for $t>m,$ we get that the LHS, by the Gauss theorem, is equal to $\frac{\Gamma(t)}{\Gamma(t-m)(\beta)_m},$ while the RHS is $(-1)^m\frac{(1-t)_m}{(\beta)_m}.$ Since they are evidently equal, we conclude the proof.
\end{proof}
Therefore, we need to prove
\begin{equation}
\sum_{j=0}^{n}A_j(1-t_j)(1-t_jt)^{-\alpha-2n-1}\leq F(n,2n+\alpha+1;1;t),
\end{equation}
where $t_j=\frac{\beta_j}{\beta_j+1}.$
One possible approach would be the usage of some properties of Hadamard products of series.

However, we will compare the appropriate coefficients on both sides of the previous inequality:
$$\sum_{j=1}^{n}A_j(1-t_j)t_j^l\leq \frac{n(n+\alpha+1)(n+l-1)!\Gamma(\alpha+l+n+1)}{l!(\alpha+1+2n)\Gamma(\alpha+2n+l+1)}.$$
We will now express the left-hand side in another way. In fact, we calculate the $l$-th derivative of $\frac{1}{P(t)}$ using two approaches. First, from
$$ \frac{1}{P(t)}=\sum_{j=1}^{n}\frac{A_j}{1+\beta_j t}$$
we find
$$  \bigg(\frac{1}{P(t)}\bigg)^{(l)}=(-1)^ll!\sum_{j=1}^{n}\frac{A_j}{1+\beta_j t}\bigg(\frac{\beta_j}{1+\beta_j t}\bigg)^l$$
and
$$\bigg(\frac{1}{P(t)}\bigg)^{(l)}\big|_{t=1}=(-1)^ll!\sum_{j=1}^{n}\frac{A_j}{1+\beta_j }\bigg(\frac{\beta_j}{1+\beta_j }\bigg)^l.$$
From this formula, we see the following equivalent form of our inequality:
$$(-1)^l\bigg(\frac{1}{P(t)}\bigg)^{(l)}\big|_{t=1}\leq \frac{n(n+{\alpha+1})(n+l-1)!\Gamma(\alpha+l+n+1)}{(\alpha+1+2n)\Gamma({\alpha}+2n+l+1)}. $$
From Leibnitz's formula we get:
\[
\begin{split}
&\bigg(\frac{1}{P(t)}\bigg)^{(l)}= \bigg(\frac{1}{(1+\beta_1t)(1+\beta_2t)\cdots(1+\beta_nt)}\bigg)^{(l)}\\
&=\sum_{\substack{j_1+j_2+\dots+j_n=l; \\ j_1,j_2,\dots,j_n\geq0}}\binom{l}{j_1;j_2;\dots;j_n} \bigg(\frac{1}{1+\beta_1t}\bigg)^{(j_1)}\bigg(\frac{1}{1+\beta_2t}\bigg)^{(j_2)}\cdots\bigg(\frac{1}{1+\beta_nt}\bigg)^{(j_n)}\\
&=\sum_{\substack{j_1+j_2+\dots+j_n=l; \\ j_1,j_2,\dots,j_n\geq0}}\binom{l}{j_1;j_2;\dots;j_n}\frac{(-1)^{j_1}j_1!\beta_1^{j_1}}{(1+\beta_1t)^{j_1+1}}\frac{(-1)^{j_2}j_2!\beta_2^{j_2}}{(1+\beta_2t)^{j_2+1}}\cdots\frac{(-1)^{j_n}j_n!\beta_n^{j_n}}{(1+\beta_nt)^{j_n+1}}\\
&=\frac{(-1)^ll!}{P(t)}\sum_{\substack{j_1+j_2+\dots+j_n=l; \\ j_1,j_2,\dots,j_n\geq0}}\frac{\beta_1^{j_1}\beta_2^{j_2}\cdots\beta_n^{j_n}}{(1+\beta_1t)^{j_1}(1+\beta_2t)^{j_2}\cdots(1+\beta_nt)^{j_n}},
\end{split}
\]
which gives
$$\sum_{j=1}^{n}A_j(1-t_j)t_j^l=\frac{1}{P(1)}\sum_{\substack{j_1+j_2+\dots+j_n=l; \\ j_1,j_2,\dots,j_n\geq0}}t_1^{j_1}t_2^{j_2}\cdots t_n^{j_n}.$$
Finally, using $\frac{1}{P(1)}=\frac{n!\Gamma(n+2+\alpha)}{\Gamma(2n+2+\alpha)},$ we reduce the inequality to showing
\begin{equation}
\label{simetricnipolinom}
\sum_{\substack{j_1+j_2+\dots+j_n=l; \\ j_1,j_2,\dots,j_n\geq0}} t_1^{j_1}t_2^{j_2}\cdots t_n^{j_n} \leq \frac{(n+l-1)!({\alpha+1}+n)_l}{l!(n-1)!({\alpha+1}+2n)_l}.
\end{equation}
If we denote the LHS of the last inequality by $S_l,$ then it takes the form
$$S_l\leq D_l=\frac{(n+l-1)!({\alpha+1}+n)_l}{l!(n-1)!({\alpha+1}+2n)_l}.$$
Using Pfaff transformation
$$P(t)=F(-1-{\alpha}-n,-n;1;t)=(1-t)^nF(-n,n+2+\alpha;1;\frac{t}{t-1}),$$
we deduce that $t_j=\frac{\beta_j}{\beta_j+1}, 1\leq j \leq n,$ are all zeros of polynomial $F(-n,n+2+\alpha;1;1-t).$
Now, we will express this function as a polynomial on $t$ and find a closed formula for the elementary symmetric polynomials of its zeros, using Vieta's formulas. First, we write it by the definition and change the order of summation:
\[
\begin{split}
&\quad\sum_{k=0}^{n}\frac{(-n)_k(n+2+\alpha)_k}{(k!)^2}(1-t)^k\\
&=\sum_{k=0}^{n}\frac{(-n)_k(n+2+\alpha)_k}{(k!)^2}\sum_{j=0}^{k}\binom{k}{j}(-1)^jt^j\\
&=\sum_{j=0}^{n}(-1)^jt^j\sum_{k=j}^{n}\frac{(-n)_k(n+2+\alpha)_k}{(k!)^2}\binom{k}{j}.
\end{split}
\]
The inner sum can be further rewritten as
\[
\begin{split}
&\quad\sum_{k=j}^{n}\frac{(-n)_k(n+2+\alpha)_k}{(k!)^2}\binom{k}{j}\\
&=\frac{(-n)_j(n+2+\alpha)_j}{(j!)^3}\sum_{k=j}^{n}\frac{(-n+j)_{k-j}(n+2+\alpha+j)_{k-j}k!}{(j+1)^2_{k-j}(k-j)!}\\
&=\frac{(-n)_j(n+2+\alpha)_j}{(j!)^2}\sum_{s=0}^{n-j}\frac{(-n+j)_s(n+2+\alpha+j)_s}{(j+1)_{s}s!}.
\end{split}
\]
Denoting $m=n-j$, we have to find the sum $\sum_{s=0}^{m}\frac{(-m)_s(2j+m+\alpha+2)_s}{(j+1)_{s}s!}$. But, by the Lemma 3.1 with $\beta=j+1$ and $l=j+{\alpha+1},$ we see that it is equal to $(-1)^m\frac{(j+\alpha+2)_m}{(1+j)_m}.$
From this we conclude that
$$\sum_{k=0}^{n}\frac{(-n)_k(n+2+\alpha)_k}{(k!)^2}(1-t)^k=(-1)^n\frac{(\alpha+2)_n}{n!}\sum_{j=0}^{n}\frac{({\alpha+2}+n)_j}{({\alpha+2})_j}(-1)^j\binom{n}{j}t^j.$$
Hence, $t_j=\frac{\beta_j}{\beta_j+1}$ are the zeros of the polynomial $\sum_{j=0}^{n}\frac{({\alpha+2}+n)_j}{({\alpha+2})_j}(-1)^j\binom{n}{j}t^j$ and by Vieta's formulas:
\begin{equation}
\label{Viet}
	\sum t_{i_1}t_{i_2}\cdots t_{i_k}=(-1)^k\frac{({\alpha+2}+n-k)_k}{(\alpha+2n-k+2)_k}\binom{n}{k}.
\end{equation}

For $n=1,$ this reduces to the Bernoulli inequality
\[
\bigg(\frac{\alpha+2}{\alpha+3}\bigg)^l \leq \frac{\alpha+2}{\alpha+l+2}.
\]
For $n=2,$ after introducing
\[
t_1 = \frac{\beta_1}{\beta_1+1} = \frac{\alpha+3 - \sqrt{\frac{2(\alpha+3)}{\alpha+4}}}{\alpha+5}, \quad
t_2 = \frac{\beta_2}{\beta_2+1} = \frac{\alpha+3 + \sqrt{\frac{2(\alpha+3)}{\alpha+4}}}{\alpha+5},
\]
we stand at:
\[
\frac{t_2^{l+1} - t_1^{l+1}}{t_2 - t_1} \leq \frac{(l+1)(\alpha+3)(\alpha+4)}{(\alpha+l+3)(\alpha+l+4)}.
\]
For $l=0$ and $l=1,$ we have equality. We will proceed by induction. If we suppose that the inequality holds for $k,$ then for $l+1$ we have:
\[
\begin{split}
\frac{t_2^{l+2} - t_1^{l+2}}{t_2 - t_1} &= \frac{t_2 (t_2^{l+1} - t_1^{l+1})}{t_2 - t_1} + t_1^{l+1} \\
&\leq \frac{(l+1)(\alpha+3)(\alpha+4)}{(\alpha+l+3)(\alpha+l+4)} t_2 + t_1^{l+1} \\
&\leq \frac{(l+2)(\alpha+3)(\alpha+4)}{(\alpha+l+4)(\alpha+l+5)}.
\end{split}
\]
The last inequality is proven by induction again. It is equality for $l=0,$ while
\[
\frac{t_1^{l+1}}{(\alpha+3)(\alpha+4)} \leq \frac{1}{\alpha+l+4} \bigg( \frac{l+2}{\alpha+l+5} - \frac{t_2(l+1)}{\alpha+l+3} \bigg)
\]
gives
\[
\begin{split}
\frac{t_1^{l+2}}{(\alpha+3)(\alpha+4)} &\leq \frac{t_1}{\alpha+l+4} \bigg( \frac{l+2}{\alpha+l+5} - \frac{t_2(l+1)}{\alpha+l+3} \bigg) \\
&\leq \frac{1}{\alpha+l+5} \bigg( \frac{l+3}{\alpha+l+6} - \frac{t_2(l+2)}{\alpha+l+4} \bigg).
\end{split}
\]

since the last inequality is equivalent to
$$(t_1+t_2)(l+2)({\alpha+l+3})({\alpha+l+6})- t_1t_2(l+1)({\alpha+l+5})_2-(l+3)({\alpha+l+3})_2\leq 0.$$
This can be easily checked using \eqref{Viet}.

Similarly, for $n=3$ we use the identity
$$S_{l+1}(t_1,t_2,t_3)=t_1S_l(t_1,t_2,t_3)+t_2S_l(t_2,t_3)+t_3^{l+1}.$$
This holds for $l=0$. If this is true for $l,$ then for $l+1,$ we have to prove:
$$t_1S_l(t_1,t_2,t_3)+t_2S_l(t_2,t_3)+t_3^{l+1}\leq t_1D_l+t_2S_l(t_2,t_3)+t_3^{l+1}\leq D_{l+1},$$
which is easily seen to be true for $l=0,$ since it reads as $t_1D_0+t_2+t_3=t_1+t_2+t_3=S_1=D_1.$
Now, to prove $t_2S_l(t_2,t_3)+t_3^{l+1}\leq D_{l+1}-t_1D_l,$ (the last inequality) we again use the induction:
\[
\begin{split}
&t_2S_{l+1}(t_2,t_3)+t_3^{l+2}=t_2(t_2S_l(t_2,t_3)+t_3^{l+1})+t_3^{l+2}\\
&\leq t_2(D_{l+1}-t_1D_l)+t_3^{l+2}\leq D_{l+2}-t_1D_{l+1},
\end{split}
\]
or, equivalently:
$$D_{l+2}-(t_1+t_2)D_{l+1}+t_1t_2D_l\geq t_3^{l+2}.$$
For $l=0$ this reads as $D_2-(t_1+t_2)D_{1}+t_1t_2D_0\geq t_3^2$ which, by using $S_0=D_0$ and $S_1=D_1,$
reduces to $D_2\geq (t_1+t_2+t_3)^2-(t_1t_2+t_2t_3+t_3t_1).$ This can be easily checked by using \eqref{Viet}.
Finally, using the induction again we see that our inequality follows from:
$$D_{l+3}-(t_1+t_2+t_3)D_{l+2}+(t_1t_2+t_2t_3+t_3t_1)D_{l+1}-t_1t_2t_3D_l\geq 0.$$

Our approach become much more complicated  and needs many modifications for $n\geq 4,$ hence we were not been able to prove the theorem in full generality for this case.

\section{Consequences to Fock space case}

Here we will show that an easy limitting argment with the measure $d\mu_{n,\alpha}$ gives the appropriate Fock space inequality from \cite{Kalaj} for each $n\geq 1$. Indeed, from
\[
\begin{split}
&\frac{k^2(k-1)^2\cdots(k-n+1)^2}{n!({\alpha+2})_n}\int_{0}^{1}\frac{t^{k-n}(1-t)^{{\alpha}+2n}}{F(-1-{\alpha}-n,-n;1;t)} d t\\
&\leq \frac{1}{{\alpha+1}+2n}\frac{k!\Gamma({\alpha+2})}{\Gamma(k+{\alpha+2})} \quad \text{for} \quad k\geq n,
\end{split}
\]
after the substitution $y=R t$ with ${(\alpha+2)}=R,$ we get:
\[
\begin{split}
&\int_{0}^{R}\frac{y^{k-n}(1-y/R)^{R+2n-2}}{F(1-R-n,-n;1;y/R)} d y\\
&\leq \frac{n!((k-n)!)^2}{k!}\frac{R^{k+1-n}(R)_n\Gamma(R)}{(R+2n-1)\Gamma(k+R)} \quad \text{for} \quad k\geq n.
\end{split}
\]

Now,
$$F(1-n-R,-n;1;y/R)=\sum_{k=0}^{n}\frac{(-n)_k(1-n-R)_k}{(k!)^2}\frac{y^k}{R^k},$$
with $R\rightarrow +\infty$ tends to
$$P_n(y)=\sum_{k=0}^{n}\frac{(-n)_k(-1)^k}{(k!)^2}y^k=L_n(-y),$$
where $L_n$ is the $n-$th degree Laguerre polynomial. Hence, taking the limit when $R\rightarrow +\infty$ in the last inequality, we get:
\[
\begin{split}
&\int_{0}^{+\infty}\frac{y^{k-n}e^{-y}}{L_n(-y)} d y\\
&\leq \frac{n!((k-n)!)^2}{k!} \quad \text{for} \quad k\geq n.
\end{split}
\]

\subsection*{Acknowledgements}
The authors would like to express their sincere gratitude to the referee for their insightful comments and constructive suggestions, which significantly enhanced the clarity and quality of this manuscript. Their expertise and attention to detail were invaluable throughout the revision process.

\end{document}